\documentclass[12pt]{amsart}
\usepackage{geometry}
\usepackage[colorlinks, citecolor = red, linkcolor=blue, hyperindex]{hyperref}
\usepackage{euscript, eufrak, verbatim, mathrsfs}
\usepackage[psamsfonts]{amssymb}
\usepackage{bbm}
\usepackage{graphicx}
\usepackage{float}
\usepackage{float, tikz}
\usepackage[all, cmtip]{xy}
\usepackage{upref, xcolor, dsfont}
\usepackage{amsfonts, amsmath, amstext, amsbsy, amsopn, amsthm}
\usepackage{enumerate}
\usepackage{url}
\usepackage{mathtools}
\usepackage{bookmark}
\usepackage{euscript}
\usepackage{helvet}         
\usepackage{courier}        
\usepackage{type1cm}        
\usepackage{multicol}        
\usepackage[bottom]{footmisc}

\newtheorem{theorem}{Theorem}[section]
\newtheorem*{theorem*}{Theorem}
\newtheorem{lemma}[theorem]{Lemma}
\newtheorem{fact}[theorem]{Fact}

\newtheorem*{comment*}{Comment}
\newtheorem*{definition*}{Definition}
\newtheorem*{remark*}{Remark}

\newtheorem*{observation*}{Observation}

\newtheorem*{assumption*}{Assumption}

\theoremstyle{definition}

\theoremstyle{remark}

\newcommand{\C}{\mathbb{C}}

\newcommand{\Conf}{\mathrm{Conf}}

\begin{document}

\title[Truncations of Hua-Pickrell unitary matrices]{Truncations of random unitary matrices drawn from Hua-Pickrell distribution}

\author
{Zhaofeng Lin}
\address{Zhaofeng Lin: Shanghai Center for Mathematical Sciences, Fudan University, Shanghai, 200438, China.}
\email{zflin18@fudan.edu.cn}

\author
{Yanqi Qiu}
\address
{Yanqi QIU: School of Mathematics and Statistics, Wuhan University, Wuhan 430072, Hubei, China; Institute of Mathematics \& Hua Loo-Keng Key Laboratory of Mathematics, AMSS, CAS, Beijing 100190, China.}
\email{yanqi.qiu@hotmail.com}

\author
{Kai Wang}
\address{Kai WANG: School of Mathematical Sciences, Fudan University, Shanghai, 200433, China.}
\email{kwang@fudan.edu.cn}

\thanks{Y.Qiu is supported by grants NSFC Y7116335K1,  NSFC 11801547 and NSFC 11688101. K. Wang is supported by grants  NSFC (12026250, 11722102) and the Shanghai Technology Innovation Project (21JC1400800).}

\begin{abstract}
Let $U$ be a random unitary matrix drawn from the Hua-Pickrell distribution $\mu_{\mathrm{U}(n+m)}^{(\delta)}$ on the unitary group $\mathrm{U}(n+m)$. We show that the eigenvalues of the truncated unitary matrix $[U_{i,j}]_{1\leq i,j\leq n}$ form a determinantal point process $\mathscr{X}_n^{(m,\delta)}$ on the unit disc $\mathbb{D}$ for any $\delta\in\mathbb{C}$ satisfying $\mathrm{Re}\,\delta>-1/2$.

We also prove that the limiting point process taken by $n\to\infty$ of the determinantal point process $\mathscr{X}_n^{(m,\delta)}$ is always $\mathscr{X}^{[m]}$, independent of $\delta$. Here $\mathscr{X}^{[m]}$ is the determinantal point process on $\mathbb{D}$ with weighted Bergman kernel
\begin{equation*}
	\begin{split}
		K^{[m]}(z,w)=\frac{1}{(1-z\overline w)^{m+1}}
	\end{split}
\end{equation*}
with respect to the reference measure $d\mu^{[m]}(z)=\frac{m}{\pi}(1-|z|)^{m-1}d\sigma(z)$, where $d\sigma(z)$ is the Lebesgue measure on $\mathbb{D}$.
\end{abstract}



\maketitle

\section{Introduction}

Determinantal point processes, which arise  in
quantum physics,  have been studied extensively after being initiated by Macchi \cite{Ma} in the seventies. They have been used to model  fermions in quantum mechanics, eigenvalues and singular  values distribution of  random matrices,  zero sets of random analytic functions, and many other objects in representation theory and combinatorics.  All of them can be described by   probabilistic models   that give the
likelihood  as    matrix determinants of   kernel functions.
We refer the reader to \cite{Bo, HKPV, PV, So, ST} for further background of determinantal point processes.

\subsection{Truncations of Haar unitary matrices}\label{Truncations of Haar unitary matrices}
By the theorem established   by Macchi \cite{Ma} and Soshnikov \cite{So}, as well as Shirai and Takahashi \cite{ST}, one  knows   how to determine whether  a reproducing kernel function  $K(\cdot,\cdot)$ yields a  determinantal point process. However, it is still a hard problem to illustrate a concrete  determinantal point process for a given kernel function, even for the classical weighted  Bergman kernel.

Let $\mathscr{X}^{[m]}$ be the  determinantal point process on the unit disc $\mathbb{D}$ with weighted Bergman kernel  \begin{equation*}
	\begin{split}
		K ^{[m]}(z,w)=\frac{1}{(1-z\overline w)^{m+1}}
	\end{split}
\end{equation*}
with respect to the reference measure
\begin{equation*}
	\begin{split}
		d\mu^{[m]}(z)=\frac{m}{\pi}(1-|z|^2)^{m-1}d\sigma(z),
	\end{split}
\end{equation*}
where $d\sigma(z)$ is the usual Lebesgue measure. 
In a breakthrough work \cite{PV}, Peres and Vir\'ag  established a concrete model  for  $\mathscr{X}^{[1]}$ through the random analytic function theory. Namely, the zeros of the random analytic function
\begin{equation*}
	\begin{split}
		f(z)=a_0+a_1z+a_2z^2+\cdots,
	\end{split}
\end{equation*}
where $a_k$, $k\geq0$, are i.i.d. standard complex Gaussian random variables, form a determinantal point process on $\mathbb{D}$ with Bergman kernel
\begin{equation*}
	\begin{split}
		K^{[1]}=\frac{1}{(1-z\overline w)^2}
	\end{split}
\end{equation*}
with respect to the normalized Lebesgue measure $d\mu^{[1]}(z)=\frac{1}{\pi}d\sigma(z)$.

Krishnapur \cite{HKPV, Kri} extended the result of Peres and Vir\'ag to general positive integer $m$, which is exactly the determinantal point process $\mathscr{X}^{[m]}$. In detail, let $G_k$, $k\geq0$, be i.i.d. $m\times m$ matrices  with i.i.d. standard complex Gaussian entries. Then the zeros of the random analytic function
\begin{equation}\label{F-Krish}
		F(z)=\det(G_0+G_1z+G_2z^2+\cdots)
\end{equation}
form a determinantal point process on $\mathbb{D}$ with the weighted Bergman kernel $K^{[m]}$ with respect to the reference measure $\mu^{[m]}$.

  The  truncated unitary matrix theory plays an important role  in Krishnapur's work.
 Let $U$ be a random unitary matrix drawn from the Haar distribution $\mu_{\mathrm{U}(n+m)}$ on unitary group $\mathrm{U}(n+m)$ for some fixed  positive integers  $n,m$. 
 Zyczkowski and Sommers \cite{ZS} showed that the eigenvalues distribution  $\mathscr{X}_n^{[m]}$ of the truncated unitary matrix $[U_{i,j}]_{1\leq i,j\leq n}$ form a determinantal point process on the unit disc $\mathbb{D}$ with kernel
\begin{equation*}
	\begin{split}
		K_n^{[m]}(z,w)=\sum_{k=0}^{n-1}\frac{(m+1)(m+2)\cdots(m+k)}{k\,!}(z\overline w)^k
	\end{split}
\end{equation*}
with respect to the reference measure $d\mu^{[m]}$. Krishnapur obtained the distribution of the zeros of the random analytic function $F$ in \eqref{F-Krish} by  establishing a  remarkable link between his model and Zyczkowski and Sommers' result.

\subsection{  Hua-Pickrell measure on unitary group} Let $\mu_{\mathrm{U}(N)}$ be
the Haar measure   on the unitary group $\mathrm{U}(N)$ for $N\in\mathbb{N_+}$; see e.g. \cite{Ha}. The celebrated theorem by Dyson and Weyl in \cite{Dy, Dy1} state that  the eigenvalues of unitary matrix drawn from the Haar distribution $\mu_{\mathrm{U}(N)}$ on $\mathrm{U}(N)$ form a determinantal point process, called   circular unitary ensemble,   on the unit circle $\mathbb{T}$ with kernel
\begin{equation*}
	\begin{split}
		K_N(e^{i\theta},e^{i\varphi})=\sum_{k=0}^{N-1}e^{ik(\theta-\varphi)}
	\end{split}
\end{equation*}
with respect to the normalized Lebesgue measure $\frac{1}{2\pi}d\theta$. We refer the reader to Pereira's short note \cite{Pe}
for  a quick survey for  more Hermitian   random matrix model, and Hedenmalm and Wennman's recent work \cite{HW} for non-Hermitian cases.

For any fixed $\delta\in\mathbb{C}$ with $\mathrm{Re}\,\delta>-1/2$, the Hua-Pickrell measure $\mu_{\mathrm{U}(N)}^{(\delta)}$ on $\mathrm{U}(N)$ is defined by
\begin{equation*}
	\begin{split}
		\mathbb{E}_{\mu_{\mathrm{U}(N)}^{(\delta)}}\big[f(U)\big]=\frac{\mathbb{E}_{\mu_{\mathrm{U}(N)}}\big[|\det(I-U)^{\delta}|^2f(U)\big]}{\mathbb{E}_{\mu_{\mathrm{U}(N)}}\big[|\det(I-U)^{\delta}|^2\big]}
	\end{split}
\end{equation*}
for any continuous function $f\in C(\mathrm{U}(N))$. In other words, the Hua-Pickrell measure $\mu_{\mathrm{U}(N)}^{(\delta)}$ is a probability measure on $\mathrm{U}(N)$ satisfying
\begin{equation*}
	\begin{split}
		d\mu_{\mathrm{U}(N)}^{(\delta)}(U)\,\propto\,\big|\det(I-U)^{\delta}\big|^2d\mu_{\mathrm{U}(N)}(U).
	\end{split}
\end{equation*}
When $\delta=0$, $\mu_{\mathrm{U}(N)}^{(\delta)}$ is just the Haar measure $\mu_{\mathrm{U}(N)}$. Here, the notation  $\mu \propto \nu$ for two measure $\mu,\nu$  means that there exists a constant $c$ such that $\mu=c \nu$.

The Hua-Pickrell measure has already been studied on the finite dimensional unitary group by Hua \cite{Hu}. And results about the infinite dimensional case were given by Pickrell \cite{Pi, Pi1}. It has been widely studied in recent years; see, e.g. \cite{Bou, BO, BNR, FW, Ne, Ol}.

Similarly, the eigenvalues of unitary matrix drawn from the Hua-Pickrell distribution $\mu_{\mathrm{U}(N)}^{(\delta)}$ on $\mathrm{U}(N)$ also form a determinantal point process on the unit circle $\mathbb{T}$, in which  the kernel $K_N^{(\delta)}$ involves  Gauss's hypergeometric functions \cite{Bou, BNR}.

\subsection{Main results}
We consider the point process for eigenvalues of the truncated unitary matrix drawn from the Hua-Pickrell distribution $\mu_{\mathrm{U}(n+m)}^{(\delta)}$ on unitary group $\mathrm{U}(n+m)$, where $n,m\in\mathbb{N_+}$ and $\delta\in\mathbb{C}$ satisfying $\mathrm{Re}\,\delta>-1/2$. Let
\begin{equation*}
	\begin{split}
		d\mu^{(m,\delta)}(z)=|(1-z)^{\delta}|^2(1-|z|^2)^{m-1}d\sigma(z)
	\end{split}
\end{equation*}
be the reference measure. Denote $\{P_k^{(m,\delta)}\}_{k=0}^{\infty}$ the family of orthonormal polynomials obtained by applying Gram-Schmidt orthogonalization procedure to $\{z^k\}_{k=0}^{\infty}$ in $L^2(\mathbb{D},\mu^{(m,\delta)})$.

\begin{theorem}\label{thm-Eigenvalues of truncations}
	Let $U$ be a random unitary matrix drawn from the Hua-Pickrell distribution $\mu_{\mathrm{U}(n+m)}^{(\delta)}$ on unitary group $\mathrm{U}(n+m)$. Then the eigenvalues of the truncated unitary matrix $[U_{i,j}]_{1\leq i,j\leq n}$ form a determinantal point process $\mathscr{X}_n^{(m,\delta)}$ on the unit disc $\mathbb{D}$ with kernel
	\begin{equation*}
		\begin{split}
			K_n^{(m,\delta)}(z,w)=\sum_{k=0}^{n-1}P_k^{(m,\delta)}(z)\overline{P_k^{(m,\delta)}(w)}
		\end{split}
	\end{equation*}
	with respect to the reference measure $\mu^{(m,\delta)}$.
\end{theorem}
  Clearly, the determinantal point process $\mathscr{X}_n^{(m,\delta)}$  is equal to the determinantal point process $\mathscr{X}_n^{[m]}$ only if $\delta=0$.
  However, we next show that the limiting point process $\mathscr{X}^{(m,\delta)}$ is always $\mathscr{X}^{[m]}$, independent of the parameter $\delta$.
\begin{theorem}\label{thm-Limiting point process}
	For any $\delta\in\mathbb{C}$ satisfying $\mathrm{Re}\,\delta>-1/2$, the limiting point process taken by $n\to\infty$ of the determinantal poin process $\mathscr{X}_n^{(m,\delta)}$ is always $\mathscr{X}^{[m]}$, independent of $\delta$. Here $\mathscr{X}^{[m]}$ is the determinantal point process on the unit disc $\mathbb{D}$ with weighted Bergman kernel
	\begin{equation*}
		\begin{split}
			K^{[m]}(z,w)=\frac{1}{(1-z\overline w)^{m+1}}
		\end{split}
	\end{equation*}
	with respect to the reference measure $d\mu^{[m]}(z)=\frac{m}{\pi}(1-|z|)^{m-1}d\sigma(z)$.
\end{theorem}

\section{Preliminaries}\label{Preliminaries}
Let us recall some notations and definitions of determinantal point processes and reproducing kernels.
\subsection{Determinantal point process}
Let $E$ be a locally compact Polish space, $\mathcal{B}_0(E)$ the collection of all pre-compact Borel subsets of $E$. Denote $\Conf(E)$ the space of all locally finite configurations over $E$, that is,
\begin{equation*}
	\begin{split}
		\Conf(E)=\big\{\xi=\textstyle\sum_{i}\delta_{x_i}\,\big|\,\,\forall i,\,x_i\in E \, \text{and $\xi(\Delta)<\infty$ for all $\Delta\in\mathcal{B}_0(E)$}\big\}.
	\end{split}
\end{equation*}
Consider the vague topology on $\Conf(E)$, the weakest topology on $\Conf(E)$ such that for any $f\in C_c(E)$, the map $\Conf(E)\ni\xi\mapsto\int_{E}fd\xi$ is continuous. Here $C_c(E)$ is the space of all continuous functions on $E$ with compact support. The configuration space $\Conf(E)$ equipped with the vague topology is also a Polish space. The Borel $\sigma$-algebra $\mathcal{F}$ on $\Conf(E)$ is generated by the cylinder sets $C_n^\Delta=\big\{\xi\in\Conf(E)\,|\,\,\xi(\Delta)=n\big\}$, where $n\in\mathbb{N}=\{0,1,2,\cdots\}$ and $\Delta\in\mathcal{B}_0(E)$. Let $(\Omega,\mathcal{F}(\Omega),\mathbb{P})$ be a probability space. By definition, a random variable
\begin{equation*}
	\begin{split}
		\mathscr{X}:(\Omega,\mathcal{F}(\Omega),\mathbb{P})\to(\Conf(E),\mathcal{F})
	\end{split}
\end{equation*}
taken integer-valued non-negative Radon measure on $E$ is called a point process on $E$. For more details,  we refer the reader to   \cite{DV,Fo, FK, HKPV, KK, Le}.

A point process $\mathscr{X}$ is called simple if it almost surely assigns at most measure one to singletons. In the simple case, $\mathscr{X}$ can be identified with a random discrete subset of $E$. And for any $\Delta\in\mathcal{B}_0(E)$, $\mathscr{X}(\Delta)$ represents the number of points that fall in $\Delta$.

Let $\mu$ be a reference Radon measure on $E$ and $K:E\times E\to\mathbb{C}$ be a measurable function. A simple point process $\mathscr{X}$ is called determinantal on $E$ associated to the kernel $K$ with respect to the reference measure $\mu$, if for every $n\in\mathbb{N}_+$ and any family of mutually disjoint subsets $\Delta_1,\Delta_2,\cdots,\Delta_m\in\mathcal{B}_0(E)$, $m\geq 1$, $n_k\geq 1$, $1\leq k\leq m$, $n_1+n_2+\cdots+n_m=n$,
\begin{equation}\label{DPP-definition}
	\begin{split}
		\mathbb{E}\Big[\prod_{k=1}^{m}\frac{\mathscr{X}(\Delta_k)\,!}{\big(\mathscr{X}(\Delta_k)-n_k\big)\,!}\Big]=\int_{\Delta_1^{n_1}\times\cdots\times\Delta_m^{n_m}} \det\big[K(x_i,x_j)\big]_{1\leq i,j\leq n}d\mu(x_1)\cdots d\mu(x_n).
	\end{split}
\end{equation}

\subsection{Orthonormal polynomials}
From now on, we always assume that $E\subset \C$ is a domain and $\mu$ be the reference Radon measure on $E$. Suppose that $\{\varphi_k\}_{k=0}^{n-1}$ is a finite orthonormal set in $L^2(E,\mu)$. Denote
\begin{equation*}
	\begin{split}
		K(x,y)=\sum_{k=0}^{n-1}\varphi_k(x)\overline{\varphi_k(y)}.
	\end{split}
\end{equation*}
Then there exists a unique determinantal point process on $E$ with kernel $K$
with respect to the reference measure $\mu$. And the number of points in this determinantal point process is equal to $n$, almost surely.

Consider a random vector in $E^n \subset \C^n$ with density
\begin{equation*}
	\begin{split}
		\frac{1}{n!}\det\big[K(x_i,x_j)\big]_{1\leq i,j\leq n}\prod_{k=1}^{n}d\mu(x_k).
	\end{split}
\end{equation*}
Erase the labels and regard it as a point process on $E$ with $n$ points, then it implies a determinantal point process on $E$ with kernel $K$ with respect to the reference measure $\mu$. The proof of these results can be found in \cite{HKPV1, HKPV}.

Let $\{P_k\}_{k=0}^{\infty}$ be a family of orthonormal polynomials obtained by applying Gram-Schmidt orthogonalization procedure to $\{x^k\}_{k=0}^{\infty}$ in $L^2(E,\mu)$. For any $n\in\mathbb{N_+}$, set
\begin{equation}\label{finite polynomials kernel}
	\begin{split}
		K_n(x,y)=\sum_{k=0}^{n-1}P_k(x)\overline{P_k(y)}.
	\end{split}
\end{equation}
It follows from the Gram-Schmidt orthogonalization procedure and the definition of Vandermonde determinant that
\begin{equation*}
	\begin{split}
		\det\big[K_n(x_i,x_j)\big]_{1\leq i,j\leq n}=\prod_{k=0}^{n-1}\mathrm{lc}^2(P_k)\prod_{1\leq i<j\leq n}|x_i-x_j|^2,
	\end{split}
\end{equation*}
where $\mathrm{lc}(P)$ denotes the leading coefficient of polynomial $P$.

Now consider a random vector in $E^n\subset \C^n$ whose density is proportional to
\begin{equation}\label{OPE}
	\begin{split}
		\prod_{1\leq i<j\leq n}|x_i-x_j|^2\prod_{k=1}^{n}d\mu(x_k).
	\end{split}
\end{equation}
It follows that the point process on $E\subset \C$ induced by erasing the labels is determinantal with kernel $K_n$ with respect to the reference measure $\mu$. The ensemble \eqref{OPE}, called the  orthogonal polynomial ensemble, has been extensively studied in random matrix models; see, e.g. \cite{BFFMPW, Dy, GMW, HKPV, Me, NW} and references therein.

\subsection{Reproducing kernel} When the kernel function $K $ appeared in the formula
(\ref{DPP-definition}) is hermitian,   it is a reproducing kernel for some Hilbert space.

Recall that a reproducing kernel Hilbert space $\mathcal{H}$ on $E$ is a Hilbert space consisted by some measurable functions satisfied that for every $x\in E$, the linear evaluation functional $ev_x:\mathcal{H}\to\mathbb{C}$ defined by $ev_x(f)=f(x)$ is bounded. Since every bounded linear functional is given by the inner product with a unique vector in $\mathcal{H}$, we know that for every $x\in E$, there is a unique vector $k_x\in\mathcal{H}$ such that $f(x)=\left\langle f,k_x\right\rangle_\mathcal{H}$ for every $f\in\mathcal{H}$. The function $k_x$ is called the reproducing kernel of the point $x$. The two-variable function $K_{\mathcal{H}}:E \times E \to\mathbb{C}$ defined by
\begin{equation*}
	\begin{split}
		K_{\mathcal{H}}(x,y)=k_y(x)
	\end{split}
\end{equation*}
is called the reproducing kernel of $\mathcal{H}$. We refer the reader to \cite{Ar, PR} for more properties and details of reproducing kernels.

It is well known that the weighted Bergman space
\begin{equation*}
	\begin{split}
		L_a^2(\mathbb{D},\mu^{[m]})=\Big\{f:\mathbb{D}\to\mathbb{C}\,\Big|\,f\,\text{is holomorphic and}\,\int_{\mathbb{D}}|f(z)|^2d\mu^{[m]}(z)<\infty\Big\}
	\end{split}
\end{equation*}
is a reproducing kernel Hilbert space with reproducing kernel
\begin{equation*}
	\begin{split}
		K^{[m]}(z,w)=\frac{1}{(1-z\overline w)^{m+1}},
	\end{split}
\end{equation*}
which is called the weighted Bergman kernel. The linear subspace  $\mathrm{span}\{z^k\}_{k=0}^{\infty}$ is dense in $L_a^2(\mathbb{D},\mu^{[m]})$, and the orthonormal polynomials have the following form:
\begin{equation*}
	\begin{split}
		P_k^{[m]}(z)=\sqrt{\frac{(m+1)(m+2)\cdots(m+k)}{k\,!}}z^k\,,\quad k\geq0.
	\end{split}
\end{equation*}
 We refer the reader to  \cite{HKZ, Zh} for a comprehensive overview.

Suppose that the reproducing kernel Hilbert space $\mathcal{H}$ is a subspace of $L^2(E,\mu)$ with reproducing kernel $K_{\mathcal{H}}(\cdot,\cdot)$. If we assume moreover that the reproducing kernel of $\mathcal{H}$ is locally of trace class, then   by a theorem obtained by Macchi \cite{Ma} and Soshnikov \cite{So}, as well as Shirai and Takahashi \cite{ST}, there exists a unique determinantal point process $\mathscr{X}_{\mathcal{H}}$ on $E$ with kernel $K_{\mathcal{H}}(\cdot,\cdot)$ with respect to the reference measure $\mu$.

In fact, the kernel given in \eqref{finite polynomials kernel} is the reproducing kernel of the finite dimensional reproducing kernel Hilbert space
\begin{equation*}
	\begin{split}
		\mathcal{H}_n=\mathrm{span}\{x^k\}_{k=0}^{n-1}\subset L^2(E,\mu).
	\end{split}
\end{equation*}
And the associated determinantal point process $\mathscr{X}_{\mathcal{H}_n}$ is exactly determined by the orthogonal polynomial ensemble~\eqref{OPE}.

Next we consider an infinite dimensional reproducing kernel Hilbert space $\mathcal{H}\subset L^2(E,\mu)$, and suppose that $\mathrm{span}\{x^k\}_{k=0}^{\infty}$ is dense in $\mathcal{H}$. Then the reproducing kernel of $\mathcal{H}$ has the following form:
\begin{equation}\label{reproducing kernel}
	\begin{split}
		K_{\mathcal{H}}(x,y)=\sum_{k=0}^{\infty}P_k(x)\overline{P_k(y)},
	\end{split}
\end{equation}
where $\{P_k\}_{k=0}^{\infty}$ is the family of orthonormal polynomials obtained by applying Gram-Schmidt orthogonalization procedure to $\{x^k\}_{k=0}^{\infty}$ in $L^2(E,\mu)$.

Hence the determinantal point process $\mathscr{X}_{\mathcal{H}}$ with kernel~\eqref{reproducing kernel} can be seen as the limiting point process taken by $n\to\infty$ of the determinantal point process $\mathscr{X}_{\mathcal{H}_n}$ with kernel~\eqref{finite polynomials kernel}. For more concrete models, more connections between determinantal point processes and reproducing kernels; see, e.g. \cite{BQ, HKPV, Kr, Kri, OS}.
\section{Proof of Theorem~\ref{thm-Eigenvalues of truncations}}
In this section, we are going to establish the point process of eigenvalues of the truncated unitary matrix drawn from the Hua-Pickrell distribution $\mu_{\mathrm{U}(n+m)}^{(\delta)}$.

 Let $\mathscr{T}(n,m)$ be
  the product measurable space
\begin{equation*}
	\begin{split}
		\mathscr{T}(n,m)=\mathrm{U}(m)\times\mathscr{V}(n)\times(S^{2m-1})^n\times\mathbb{C}^n,
	\end{split}
\end{equation*}
where $\mathscr{V}(n)$ is the submanifold of $n\times n$ unitary matrices with non-negative diagonal elements and $S^{2m-1}$ is the $(2m-1)$-dimensional unit sphere. There exists  a   product probability measure $\mathcal{T}_{(n,m)}$ on $\mathscr{T}(n,m)$  defined by
\begin{equation*}
	\begin{split}
		&\quad\,\,d\mathcal{T}_{(n,m)}(W,V,\omega_1,\cdots,\omega_n,z_1,\cdots,z_n)\\
		&=C_{(n,m)}\prod_{k=1}^{n}(1-|z_k|^2)^{m-1}\prod_{1\leq i<j\leq n}|z_i-z_j|^2\,d\mu_{\mathrm{U}(m)}(W)\,d\nu_n(V)\prod_{k=1}^{n}d\Theta(\omega_k)\prod_{k=1}^{n}d\sigma(z_k),
	\end{split}
\end{equation*}
where $C_{(n,m)}$ is the normalization constant, $d\nu_n$ is the restriction of Haar measure $d\mu_{\mathrm{U}(n)}$ on the submanifold $\mathscr{V}(n)$ and $d\Theta$ is the usual Lebesgue measure on $S^{2m-1}$.

Let  $\mathcal{P}:\mathscr{T}(n,m)\to\mathbb{C}^n$ be the projection map.
The following decomposition for the Haar measure  $\mu_{\mathrm{U}(n+m)}$ on unitary group $\mathrm{U}(n+m)$ may be well-known. Since the lack of references,  we present a full proof for completeness.

\begin{lemma}\label{lemma about Haar measure}
	There exists a measurable transformation
	\begin{equation*}
		\begin{split}
			T:\big(\mathrm{U}(n+m),\mu_{\mathrm{U}(n+m)}\big)\to\big(\mathscr{T}(n,m),\mathcal{T}_{(n,m)}\big),
		\end{split}
	\end{equation*}
	which possesses the following two properties $:$
	
	$(\mathrm{i}).$ for $\mu_{\mathrm{U}(n+m)}$-a.e. $U\in\mathrm{U}(n+m)$, the set $\{(\mathcal{P}T(U))_i:1\leqslant i\leqslant n\} $ is equal to the set $\lambda\big([U_{i,j}]_{1\leq i,j\leq n}\big)$ of  eigenvalues of the matrix $[U_{i,j}]_{1\leq i,j\leq n};$
	
	$(\mathrm{ii}).$ for any integrable function $\varphi\in L^1\big(\mathscr{T}(n,m),\mathcal{T}_{(n,m)}\big)$,
	\begin{equation*}
		\begin{split}
			\int_{\mathrm{U}(n+m)}\varphi\circ T\,d\mu_{\mathrm{U}(n+m)}=\int_{\mathscr{T}(n,m)}\varphi\,d\mathcal{T}_{(n,m)}.
		\end{split}
	\end{equation*}
\end{lemma}

\begin{proof}[\bf Proof of Lemma~\ref{lemma about Haar measure}]
We will prove Lemma~\ref{lemma about Haar measure} by the following five steps.

\textbf{Step I.}
For a matrix $M\in\mathrm{GL}(n+m,\mathbb{C})$, partition it as
\begin{equation}\label{partitioned matrix}
	\begin{split}
		M=
		\begin{bmatrix}
			X&C\\
			B^*&A
		\end{bmatrix},
	\end{split}
\end{equation}
where $X$ has size $n\times n$. By omitting a lower dimensional submanifold of $\mathrm{GL}(n+m,\mathbb{C})$, whose Lebesgue measure is zero, we may assume that the eigenvalues of $X$ are mutually distinct.
 Then we can put the eigenvalues of $X$ in lexicographical order by $z_1<z_2<\dots<z_n$.

The submatrix $X$ can be uniquely written as Schur decomposition
\begin{equation}\label{Schur decomposition}
	\begin{split}
		X=V(Z+T)V^*,
	\end{split}
\end{equation}
where $V$ is unitary with $V_{i,i}\geq0$, $1\leq i\leq n$, $T$ is strictly upper triangular, $Z$ is the diagonal matrix $\mathrm{diag}(z_1,z_2,\dots,z_n)$; see e.g. \cite{HJ}. Therefore, there is a unique matrix decomposition of $M$ as follow:
\begin{equation}\label{modified partitioned matrix}
	\begin{split}
		M=
		\begin{bmatrix}
			V&0\\
			0&I
		\end{bmatrix}
		\begin{bmatrix}
			Z+T&S\\
			R^*&A
		\end{bmatrix}
		\begin{bmatrix}
			V^*&0\\
			0&I
		\end{bmatrix},
	\end{split}
\end{equation}
where $R=V^*B$ and $S=V^*C$.

Write   $R=\big[\mathcal{R}_1^{\mathrm{T}},\mathcal{R}_2^{\mathrm{T}},\cdots,\mathcal{R}_n^{\mathrm{T}}\big]^{\mathrm{T}}$, where each $\mathcal{R}_k$ is a $m$-tuple row vector. For $1\leq k\leq n$, let $R_k=\big[\mathcal{R}_1^{\mathrm{T}},\mathcal{R}_2^{\mathrm{T}},\cdots,\mathcal{R}_k^{\mathrm{T}}\big]^{\mathrm{T}}$ denote the submatrix consisting of the first $k$ rows of $R$, and $\mathscr{A}_k$ be the submatrix consisting of the first $k$ rows and columus of matrix $Z+T$. In particular, $R_1=\mathcal{R}_1$, $R_n=R$, and $\mathscr{A}_1=z_1$, $\mathscr{A}_n=Z+T$. For $1\leq k\leq n$, since $I+(R_k^*\mathscr{A}_k^{-1})(R_k^*\mathscr{A}_k^{-1})^*$ is a positive definite matrix, the positive definite square root $\big[I+(R_k^*\mathscr{A}_k^{-1})(R_k^*\mathscr{A}_k^{-1})^*\big]^{1/2}$ makes sense. Denote
\begin{equation}\label{change by positive definite square root}
	\begin{split}
		\left\{\begin{array}{l}
			\widetilde{\mathcal{R}}_1=\mathcal{R}_1,\\
			\widetilde{\mathcal{R}}_{k+1}=\mathcal{R}_{k+1}\big[I+(R_k^*\mathscr{A}_k^{-1})(R_k^*\mathscr{A}_k^{-1})^*\big]^{1/2},\quad k=1,2,\cdots,n-1,\\
			\widetilde{A}=\big[I+(R_n^*\mathscr{A}_n^{-1})(R_n^*\mathscr{A}_n^{-1})^*\big]^{1/2}A.
		\end{array}\right.
	\end{split}
\end{equation}

For $k=1,2,\cdots,n$, write $\widetilde{\mathcal{R}}_{k}=r_k\omega_k$ in the spherical coordinate system, where the $m$-tuple row vector $\omega_k$ lies in the $(2m-1)$-dimensional unit sphere $S^{2m-1}$ and $r_k$ is a non-negative number. Let $\widetilde{A}=WD^{1/2}$ be the matrix's polar decomposition of $\widetilde{A}$, i.e., $W$ lies in the $m\times m$  unitary group $\mathrm{U}(m)$ and $D^{1/2}$ is the positive definite square root of a positive definite matrix $D$.

We have    the following measure decomposition for Lebesgue measure on the general linear group $\mathrm{GL}(n+m,\mathbb{C})$.

\begin{fact}\label{measure decomposition for general linear group}
	In the  above notations, the Lebesgue measure on $\mathrm{GL}(n+m,\mathbb{C})$ has the following decomposition:
	\begin{equation*}
		\begin{split}
			\bigwedge_{1\leq i,j\leq n+m}|dM_{i,j}|^2&\,\propto\,\frac{f(D)\prod_{k=1}^{n}r_k^{2m-2}\prod_{1\leq i<j\leq n}|z_i-z_j|^2}{\det[I+(R_n^*\mathscr{A}_n^{-1})(R_n^*\mathscr{A}_n^{-1})^*]^m\prod_{k=1}^{n-1}\det[I+(R_k^*\mathscr{A}_k^{-1})(R_k^*\mathscr{A}_k^{-1})^*]}\\
			&\quad\cdot\bigwedge_{1\leq k\leq n}d\sigma(z_k)\bigwedge_{1\leq k\leq n}d\Theta(\omega_k)\bigwedge d\nu_n(V)\bigwedge d\mu_{\mathrm{U}(m)}(W)\\
			&\quad\cdot\bigwedge_{1\leq i<j\leq n}|dT_{i,j}|^2\bigwedge_{\substack{1\leq i\leq n\\1\leq j\leq m}}|dS_{i,j}|^2\bigwedge_{1\leq i,j\leq m}dD_{i,j}\bigwedge_{1\leq k\leq n}r_kdr_k,
		\end{split}
	\end{equation*}
	where $f$ is a smooth function of $D$.
\end{fact}

Proof of Fact~\ref{measure decomposition for general linear group}:
	It follows from the partitioned matrix form~\eqref{partitioned matrix} that
	\begin{equation*}
		\begin{split}
			\bigwedge_{1\leq i,j\leq n+m}|dM_{i,j}|^2=\bigwedge_{1\leq i,j\leq n}|dX_{i,j}|^2\bigwedge_{\substack{1\leq i\leq n\\1\leq j\leq m}}|dC_{i,j}|^2\bigwedge_{\substack{1\leq i\leq n\\1\leq j\leq m}}|dB_{i,j}|^2\bigwedge_{1\leq i,j\leq m}|dA_{i,j}|^2.
		\end{split}
	\end{equation*}
	Noticing that $R=V^*B$, $S=V^*C$ and $V$ is unitary, we have
	\begin{equation*}
		\begin{split}
			\bigwedge_{\substack{1\leq i\leq n\\1\leq j\leq m}}|dB_{i,j}|^2=\bigwedge_{\substack{1\leq i\leq n\\1\leq j\leq m}}|dR_{i,j}|^2,
		\end{split}
	\end{equation*}
	and
	\begin{equation*}
		\begin{split}
			\bigwedge_{\substack{1\leq i\leq n\\1\leq j\leq m}}|dC_{i,j}|^2=\bigwedge_{\substack{1\leq i\leq n\\1\leq j\leq m}}|dS_{i,j}|^2.
		\end{split}
	\end{equation*}
	Therefore,
	\begin{equation}\label{A}
		\begin{split}
			\bigwedge_{1\leq i,j\leq n+m}|dM_{i,j}|^2=\bigwedge_{1\leq i,j\leq n}|dX_{i,j}|^2\bigwedge_{\substack{1\leq i\leq n\\1\leq j\leq m}}|dS_{i,j}|^2\bigwedge_{\substack{1\leq i\leq n\\1\leq j\leq m}}|dR_{i,j}|^2\bigwedge_{1\leq i,j\leq m}|dA_{i,j}|^2.
		\end{split}
	\end{equation}
	
Applying the Ginibre's measure decomposition in Page 105 in \cite{HKPV} to
	the matrix's Schur decomposition~\eqref{Schur decomposition}, we have
	\begin{equation}\label{B}
		\begin{split}
			\bigwedge_{1\leq i,j\leq n}|dX_{i,j}|^2\,\propto\,\prod_{1\leq i<j\leq n}|z_i-z_j|^2\bigwedge_{1\leq k\leq n}d\sigma(z_k)\bigwedge d\nu_n(V)\bigwedge_{1\leq i<j\leq n}|dT_{i,j}|^2.
		\end{split}
	\end{equation}
	
We next turn to the computation of $\bigwedge\limits_{\substack{1\leq i\leq n\\1\leq j\leq m}}|dR_{i,j}|^2$ and $\bigwedge\limits_{1\leq i,j\leq m}|dA_{i,j}|^2$.
	The formulas~\eqref{change by positive definite square root} implies that
	\begin{equation*}
		\begin{split}
			\bigwedge_{\substack{1\leq i\leq n\\1\leq j\leq m}}|dR_{i,j}|^2=\prod_{k=1}^{n-1}\det[I+(R_k^*\mathscr{A}_k^{-1})(R_k^*\mathscr{A}_k^{-1})^*]^{-1}\bigwedge_{\substack{1\leq i\leq n\\1\leq j\leq m}}|d\widetilde{R}_{i,j}|^2,
		\end{split}
	\end{equation*}
	and
	\begin{equation*}
		\begin{split}
			\bigwedge_{1\leq i,j\leq m}|dA_{i,j}|^2=\det[I+(R_n^*\mathscr{A}_n^{-1})(R_n^*\mathscr{A}_n^{-1})^*]^{-m}\bigwedge_{1\leq i,j\leq m}|d\widetilde{A}_{i,j}|^2.
		\end{split}
	\end{equation*}
	
	Based on the spherical coordinate systems $\widetilde{\mathcal{R}}_{k}=r_k\omega_k$, $k=1,2,\cdots,n$, we have
	\begin{equation*}
		\begin{split}
			\bigwedge_{\substack{1\leq i\leq n\\1\leq j\leq m}}|d\widetilde{R}_{i,j}|^2=\prod_{k=1}^{n}r_k^{2m-2}\bigwedge_{1\leq k\leq n}r_kdr_k\bigwedge_{1\leq k\leq n}d\Theta(\omega_k),
		\end{split}
	\end{equation*}
	and
	\begin{equation}\label{C}
		\begin{split}
			\bigwedge_{\substack{1\leq i\leq n\\1\leq j\leq m}}|dR_{i,j}|^2=\frac{\prod_{k=1}^{n}r_k^{2m-2}\bigwedge_{1\leq k\leq n}r_kdr_k\bigwedge_{1\leq k\leq n}d\Theta(\omega_k)}{\prod_{k=1}^{n-1}\det[I+(R_k^*\mathscr{A}_k^{-1})(R_k^*\mathscr{A}_k^{-1})^*]}.
		\end{split}
	\end{equation}
	
	Moreover, by  Lemma 6.6.1 in \cite {HKPV}, the matrix's polar decomposition $\widetilde{A}=WD^{1/2}$ yields that
	\begin{equation*}
		\begin{split}
			\bigwedge_{1\leq i,j\leq m}|d\widetilde{A}_{i,j}|^2=f(D)\bigwedge_{1\leq i,j\leq m}dD_{i,j}\bigwedge_{1\leq i, j\leq m}\Omega_{i,j}(W),
		\end{split}
	\end{equation*}
	where $f$ is a smooth function of $D$ and $\Omega(W)=W^*dW$ is a matrix-valued one form.  It also follows from the statement in  Page 101 of \cite{HKPV}  that
	\begin{equation*}
		\begin{split}
			d\mu_{\mathrm{U}(m)}(W)\,\propto\,\bigwedge_{1\leq i, j\leq m}\Omega_{i,j}(W).
		\end{split}
	\end{equation*}
	This concludes that
	\begin{equation}\label{D}
		\begin{split}
			\bigwedge_{1\leq i,j\leq m}|dA_{i,j}|^2\,\propto\,\frac{f(D)\bigwedge_{1\leq i,j\leq m}dD_{i,j}\bigwedge d\mu_{\mathrm{U}(m)}(W)}{\det[I+(R_n^*\mathscr{A}_n^{-1})(R_n^*\mathscr{A}_n^{-1})^*]^m}.
		\end{split}
	\end{equation}
	
	Combining \eqref{B}, \eqref{C}, \eqref{D} with \eqref{A}, we obtain Fact~\ref{measure decomposition for general linear group}.

\textbf{Step II.} For a matrix $M\in GL(n+m)$,
write $M=UP^{1/2}$ in the polar decomposition, where $U$ lies in the unitary group $\mathrm{U}(n+m)$ and $P^{1/2}$ is the positive definite square root of a positive definite matrix $P$. We denote a normal matrix $Q$ by
\begin{equation}\label{N}
	\begin{split}
		Q=
		\begin{bmatrix}
			V^*&0\\
			0&I
		\end{bmatrix}
		P
		\begin{bmatrix}
			V&0\\
			0&I
		\end{bmatrix},
	\end{split}
\end{equation}
and partition it as
\begin{equation}\label{P}
	\begin{split}
		Q=
		\begin{bmatrix}
			Q^{(1)}&Q^{(2)}\\
			Q^{(3)}&Q^{(4)}
		\end{bmatrix},
	\end{split}
\end{equation}
where $Q^{(1)}$ has size $n\times n$. Here $V$ is the matrix appeared in  the matrix's Schur decomposition~\eqref{Schur decomposition}

By further studying for Fact~\ref{measure decomposition for general linear group}, we have the following measure decomposition for Lebesgue measure on the general linear group $\mathrm{GL}(n+m,\mathbb{C})$.

\begin{fact}\label{new measure decomposition for general linear group}
	In the  above notations, the Lebesgue measure on $\mathrm{GL}(n+m,\mathbb{C})$ has the following decomposition:
	\begin{equation*}
		\begin{split}
			\bigwedge_{1\leq i,j\leq n+m}|dM_{i,j}|^2&\,\propto\,\frac{f(D)\prod_{k=1}^{n}r_k^{2m-2}\prod_{1\leq i<j\leq n}|z_i-z_j|^2}{\det(\mathscr{A}_n^*\mathscr{A}_n+R_nR_n^*)^m\prod_{k=1}^{n-1}\det(\mathscr{A}_k^*\mathscr{A}_k+R_kR_k^*)}\\
			&\quad\cdot\bigwedge_{1\leq k\leq n}d\sigma(z_k)\bigwedge_{1\leq k\leq n}d\Theta(\omega_k)\bigwedge d\nu_n(V)\bigwedge d\mu_{\mathrm{U}(m)}(W)\\
			&\quad\cdot\bigwedge_{\substack{1\leq i,j\leq n\\i\neq j}}dQ^{(1)}_{i,j}\bigwedge_{\substack{1\leq i\leq n\\1\leq j\leq m}}dQ^{(2)}_{i,j}\bigwedge_{\substack{1\leq i\leq m\\1\leq j\leq n}}dQ^{(3)}_{i,j}\bigwedge_{1\leq i,j\leq m}dD_{i,j}\bigwedge_{1\leq k\leq n}r_kdr_k.
		\end{split}
	\end{equation*}
\end{fact}

Proof of Fact~\ref{new measure decomposition for general linear group}:
	It is easy to verify that
	\begin{equation}\label{E}
		\begin{split}
			\begin{bmatrix}
				(Z+T)^*&R\\
				S^*&A^*
			\end{bmatrix}
			\begin{bmatrix}
				Z+T&S\\
				R^*&A
			\end{bmatrix}
			=Q.
		\end{split}
	\end{equation}
	This implies that
	\begin{equation*}
		\begin{split}
			(Z+T)^*S+RA=Q^{(2)}.
		\end{split}
	\end{equation*}
	Noticing that
	\begin{equation}\label{H}
		\begin{split}
			A=\big[I+(R_n^*\mathscr{A}_n^{-1})(R_n^*\mathscr{A}_n^{-1})^*\big]^{-1/2}WD^{1/2},
		\end{split}
	\end{equation}
	we have
	\begin{equation}\label{I}
		\begin{split}
			S=\mathscr{A}_n^{*-1}\Big(Q^{(2)}-R_n\big[I+(R_n^*\mathscr{A}_n^{-1})(R_n^*\mathscr{A}_n^{-1})^*\big]^{-1/2}WD^{1/2}\Big).
		\end{split}
	\end{equation}
	Based on \eqref{change by positive definite square root}, one can prove by induction that $R_n$ is depended on $Z$, $T$, $r_1,r_2,\cdots,r_n$, $\omega_1,\omega_2,\cdots,\omega_n$, and $\mathscr{A}_n$ is depended on $Z$, $T$. This implies that
	\begin{equation*}
		\begin{split}
			\bigwedge_{\substack{1\leq i\leq n\\1\leq j\leq m}}dS_{i,j}=\frac{1}{(\det\mathscr{A}_n^*)^m}\bigwedge_{\substack{1\leq i\leq n\\1\leq j\leq m}}dQ^{(2)}_{i,j}+[\cdots],
		\end{split}
	\end{equation*}
	where $[\cdots]$ consists of many terms involving $dz_k,dr_k,d\omega_k$, $1\leq k\leq n$, and $dD_{i,j},dW_{i,j}$, $1\leq i,j\leq m$, as well as $dT_{i,j}$, $1\leq i<j\leq n$.
	Noting that $Q^{(2)*}=Q^{(3)}$, we get
	\begin{equation}\label{F}
		\begin{split}
			\bigwedge_{\substack{1\leq i\leq n\\1\leq j\leq m}}|dS_{i,j}|^2=\frac{1}{\det(\mathscr{A}_n^*\mathscr{A}_n)^m}\bigwedge_{\substack{1\leq i\leq n\\1\leq j\leq m}}dQ^{(2)}_{i,j}\bigwedge_{\substack{1\leq i\leq m\\1\leq j\leq n}}dQ^{(3)}_{i,j}+[\cdots].
		\end{split}
	\end{equation}
	
	For $1\leq k\leq n-1$, denote $k$-tuple columu vectors
	\begin{equation*}
		\begin{split}
			t_{k+1}=\big[T_{1,k+1},T_{2,k+1},\cdots,T_{k,k+1}\big]^{\mathrm{T}},
		\end{split}
	\end{equation*}
	and
	\begin{equation}\label{Z}
		\begin{split}
			q_{k+1}=\big[Q^{(1)}_{1,k+1},Q^{(1)}_{2,k+1},\cdots,Q^{(1)}_{k,k+1}\big]^{\mathrm{T}}.
		\end{split}
	\end{equation}
	It follows from \eqref{E} that
	\begin{equation*}
		\begin{split}
			\mathscr{A}_k^*t_{k+1}+R_k\mathcal{R}_{k+1}^*=q_{k+1}.
		\end{split}
	\end{equation*}
	By the fact that
	\begin{equation}\label{K}
		\begin{split}
			\mathcal{R}_{k+1}=r_{k+1}\omega_{k+1}\big[I+(R_k^*\mathscr{A}_k^{-1})(R_k^*\mathscr{A}_k^{-1})^*\big]^{-1/2},
		\end{split}
	\end{equation}
	we get
	\begin{equation}\label{J}
		\begin{split}
			t_{k+1}=\mathscr{A}_k^{*-1}\Big(q_{k+1}-r_{k+1}R_k\big[I+(R_k^*\mathscr{A}_k^{-1})(R_k^*\mathscr{A}_k^{-1})^*\big]^{-1/2}\omega_{k+1}^*\Big).
		\end{split}
	\end{equation}
	 Using the fact that $R_k$ is depended on $z_1,z_2,\cdots,z_k$, $r_1,r_2,\cdots,r_k$, $\omega_1,\omega_2,\cdots,\omega_k$ and $T_{i,j}$, $1\leq i<j\leq k$, and $\mathscr{A}_k$ is depended on $z_1,z_2,\cdots,z_k$, $T_{i,j}$ again,   we  obtain that
	\begin{equation*}
		\begin{split}
			\bigwedge_{1\leq i\leq k}dT_{i,k+1}=\frac{1}{\det\mathscr{A}_k^*}\bigwedge_{1\leq i\leq k}dQ^{(1)}_{i,k+1}+[\cdots],
		\end{split}
	\end{equation*}
	where $[\cdots]$ consists of many terms involving $dz_j$, $1\leq j\leq k$, and $dr_j,d\omega_j$, $1\leq j\leq k+1$, as well as $dT_{i,j}$, $1\leq i<j\leq k$. Noting that $Q^{(1)*}=Q^{(1)}$, we get
	\begin{equation}\label{G}
		\begin{split}
			\bigwedge_{1\leq i\leq k}|dT_{i,k+1}|^2=\frac{1}{\det(\mathscr{A}_k^*\mathscr{A}_k)}\bigwedge_{1\leq i\leq k}dQ^{(1)}_{i,k+1}\bigwedge_{1\leq j\leq k}dQ^{(1)}_{k+1,j}+[\cdots].
		\end{split}
	\end{equation}
	
A direct computation shows that for any $1\leq k\leq n$,
	\begin{equation*}
		\begin{split}
			&\quad\,\det\big(\mathscr{A}_k^*\mathscr{A}_k\big)\det\big[I+(R_k^*\mathscr{A}_k^{-1})(R_k^*\mathscr{A}_k^{-1})^*\big]\\
			&=\det\big(\mathscr{A}_k^*\mathscr{A}_k\big)\det\big[I+R_k^*(\mathscr{A}_k^*\mathscr{A}_k)^{-1}R_k\big]\\
			&=\det\big(\mathscr{A}_k^*\mathscr{A}_k\big)\det\big[I+(\mathscr{A}_k^*\mathscr{A}_k)^{-1}R_kR_k^*\big]\\
			&=\det\big(\mathscr{A}_k^*\mathscr{A}_k+R_kR_k^*\big).
		\end{split}
	\end{equation*}
Combing it with \eqref{F}, \eqref{G} and Fact~\ref{measure decomposition for general linear group}, we obtain  	  Fact~\ref{new measure decomposition for general linear group}.

\textbf{Step III.}
Denote
\begin{equation}\label{Q}
	\begin{split}
		G=Q^{(2)*}\big(\mathscr{A}_n^*\mathscr{A}_n\big)^{-1}R_n\big[I+(R_n^*\mathscr{A}_n^{-1})(R_n^*\mathscr{A}_n^{-1})^*\big]^{-1/2}W,
	\end{split}
\end{equation}
and for $1\leq k\leq n-1$,
\begin{equation}\label{R}
	\begin{split}
		\alpha_k=\mathrm{Re}\Big(\omega_{k+1}\big[I+(R_k^*\mathscr{A}_k^{-1})(R_k^*\mathscr{A}_k^{-1})^*\big]^{-1/2}R_k^*\big(\mathscr{A}_k^*\mathscr{A}_k\big)^{-1}q_{k+1}\Big).
	\end{split}
\end{equation}

By further studying for Fact~\ref{new measure decomposition for general linear group}, we have the following measure decomposition for Lebesgue measure on the general linear group $\mathrm{GL}(n+m,\mathbb{C})$.

\begin{fact}\label{new new measure decomposition for general linear group}
	In the   above notations, the Lebesgue measure on $\mathrm{GL}(n+m,\mathbb{C})$ has the following decomposition:
	\begin{equation*}
		\begin{split}
			\bigwedge_{1\leq i,j\leq n+m}|dM_{i,j}|^2&\,\propto\,\frac{f(D)g(G,D)\prod_{k=1}^{n}r_k^{2m-2}\prod_{1\leq i<j\leq n}|z_i-z_j|^2}{\det(\mathscr{A}_n^*\mathscr{A}_n+R_nR_n^*)^m\prod_{k=1}^{n-1}\det(\mathscr{A}_k^*\mathscr{A}_k+R_kR_k^*)\prod_{k=1}^{n-1}(1-\alpha_kr_{k+1}^{-1})}\\
			&\quad\cdot\bigwedge_{1\leq k\leq n}d\sigma(z_k)\bigwedge_{1\leq k\leq n}d\Theta(\omega_k)\bigwedge d\nu_n(V)\bigwedge d\mu_{\mathrm{U}(m)}(W)\bigwedge_{1\leq i,j\leq n+m}dQ_{i,j},
		\end{split}
	\end{equation*}
	where $g$ is a smooth function of $(G,D)$.
\end{fact}

Proof of Fact~\ref{new new measure decomposition for general linear group}:
	It follows from \eqref{E} that
	\begin{equation*}
		\begin{split}
			S^*S+A^*A=Q^{(4)}.
		\end{split}
	\end{equation*}
	Based on \eqref{H} and \eqref{I}, we have
	\begin{equation}\label{S}
		\begin{split}
			D-GD^{1/2}-D^{1/2}G^*+Q^{(2)*}\big(\mathscr{A}_n^*\mathscr{A}_n\big)^{-1}Q^{(2)}=Q^{(4)}.
		\end{split}
	\end{equation}
	This implies that
	\begin{equation}\label{L}
		\begin{split}
			\bigwedge_{1\leq i,j\leq m}dD_{i,j}=g(G,D)\bigwedge_{1\leq i,j\leq m}dQ^{(4)}_{i,j}+[\cdots],
		\end{split}
	\end{equation}
	where $g$ is a smooth function of $(G,D)$. 
Here $[\cdots]$ consists of many terms involving $dz_k,d\omega_k$, $1\leq k\leq n$, $dW_{i,j}$, $1\leq i,j\leq m$, and $dQ^{(2)}_{i,j}$, $1\leq i\leq n$, $1\leq j\leq m$, $dQ^{(1)}_{i,j}$, $1\leq i,j\leq n$, $i\ne j$, as well as $dr_k$, $1\leq k\leq n$.
	
	It follows from \eqref{E} that
	\begin{equation*}
		\begin{split}
			\left\{\begin{array}{l}
				|z_1|^2+\mathcal{R}_1\mathcal{R}_1^*=Q^{(1)}_{1,1},\\
				t_{k+1}^*t_{k+1}+|z_{k+1}|^2+\mathcal{R}_{k+1}\mathcal{R}_{k+1}^*=Q^{(1)}_{k+1,k+1},\quad k=1,2,\cdots,n-1.
			\end{array}\right.
		\end{split}
	\end{equation*}
	Based on \eqref{K} and \eqref{J}, we have
	\begin{equation}\label{T}
		\begin{split}
			\left\{\begin{array}{l}
				r_1^2=Q^{(1)}_{1,1}-|z_1|^2,\\
				r_{k+1}^2-2\alpha_kr_{k+1}+q_{k+1}^*\big(\mathscr{A}_k^*\mathscr{A}_k\big)^{-1}q_{k+1}=Q^{(1)}_{k+1,k+1}-|z_{k+1}|^2,\quad k=1,2,\cdots,n-1.
			\end{array}\right.
		\end{split}
	\end{equation}
	This implies that
	\begin{equation}\label{M}
		\begin{split}
			\left\{\begin{array}{l}
				2r_1dr_1=dQ^{(1)}_{1,1}-\overline{z}_1dz_1-z_1d\overline{z}_1,\\
				2r_{k+1}dr_{k+1}=(1-\alpha_kr_{k+1}^{-1})^{-1}dQ^{(1)}_{k+1,k+1}+[\cdots]_{k+1},\quad k=1,2,\cdots,n-1.
			\end{array}\right.
		\end{split}
	\end{equation}
 Here $[\cdots]_{k+1}$ consists of many terms involving $dz_j$, $1\leq j\leq k$, $d\omega_j$, $1\leq j\leq k+1$, $dQ^{(1)}_{i,j}$, $1\leq i,j\leq k+1$, $i\ne j$, as well as $dr_j$, $1\leq j\leq k$.
	
Combing Fact~\ref{new measure decomposition for general linear group} with	\eqref{L} and \eqref{M}, we obtain    Fact~\ref{new new measure decomposition for general linear group}.

\textbf{Step IV.}
By further studying for Fact~\ref{new new measure decomposition for general linear group}, we have the following measure decomposition for Lebesgue measure on the general linear group $\mathrm{GL}(n+m,\mathbb{C})$.

\begin{fact}\label{new new new measure decomposition for general linear group}
	In the  above notations, the Lebesgue measure on $\mathrm{GL}(n+m,\mathbb{C})$ has the following decomposition:
	\begin{equation*}
		\begin{split}
			\bigwedge_{1\leq i,j\leq n+m}|dM_{i,j}|^2&\,\propto\,\frac{f(D)g(G,D)\prod_{k=1}^{n}r_k^{2m-2}\prod_{1\leq i<j\leq n}|z_i-z_j|^2}{\det(\mathscr{A}_n^*\mathscr{A}_n+R_nR_n^*)^m\prod_{k=1}^{n-1}\det(\mathscr{A}_k^*\mathscr{A}_k+R_kR_k^*)\prod_{k=1}^{n-1}(1-\alpha_kr_{k+1}^{-1})}\\
			&\quad\cdot\bigwedge_{1\leq k\leq n}d\sigma(z_k)\bigwedge_{1\leq k\leq n}d\Theta(\omega_k)\bigwedge d\nu_n(V)\bigwedge d\mu_{\mathrm{U}(m)}(W)\bigwedge_{1\leq i,j\leq n+m}dP_{i,j}.
		\end{split}
	\end{equation*}
\end{fact}

Proof of Fact~\ref{new new new measure decomposition for general linear group}:
	Based on \eqref{N}, we have
	\begin{equation}\label{O}
		\begin{split}
			\bigwedge_{1\leq i,j\leq n+m}dQ_{i,j}=\bigwedge_{1\leq i,j\leq n+m}dP_{i,j}+[\cdots],
		\end{split}
	\end{equation}
	where $[\cdots]$ consists of many terms involving $dV_{i,j}$, $1\leq i, j\leq n$. Notice that the dimension of the submanifold consisted by $n\times n$ unitary matrices with non-negative diagonal elements is $n^2-n$. It follows from Claim~6.3.1 in \cite{HKPV} that for any $1\leq i, j\leq n$,
	\begin{equation*}
		\begin{split}
			d\nu_n(V)\bigwedge dV_{i,j}=0.
		\end{split}
	\end{equation*}
	Hence putting \eqref{O} into Fact~\ref{new new measure decomposition for general linear group}, we obtain Fact~\ref{new new new measure decomposition for general linear group}.

\textbf{Step V.}
We now turn to consider the Haar measure $\mu_{\mathrm{U}(n+m)}$ on the unitary group $\mathrm{U}(n+m)$. We have the following measure decomposition.

\begin{fact}\label{Haar measure decomposition}
	In the above notations, the Haar measure $\mu_{\mathrm{U}(n+m)}$ on $\mathrm{U}(n+m)$ has the following decomposition:
	\begin{equation*}
		\begin{split}
			d\mu_{\mathrm{U}(n+m)}(U)\,\propto\,\prod_{k=1}^{n}(1-|z_k|^2)^{m-1}\prod_{1\leq i<j\leq n}|z_i-z_j|^2\,d\mu_{\mathrm{U}(m)}(W)\,d\nu_n(V)\prod_{k=1}^{n}d\Theta(\omega_k)\prod_{k=1}^{n}d\sigma(z_k).
		\end{split}
	\end{equation*}
\end{fact}

Proof of Fact~\ref{Haar measure decomposition}:
	The matrix's polar decomposition $M=UP^{1/2}$ yields that
	\begin{equation}\label{U}
		\begin{split}
			\bigwedge_{1\leq i,j\leq n+m}|dM_{i,j}|^2=h(P)\,d\mu_{\mathrm{U}(n+m)}(U)\bigwedge_{1\leq i,j\leq m}dP_{i,j},
		\end{split}
	\end{equation}
	where $h$ is a smooth function of $P$.
	
	Based on the polar decomposition $M=UP^{1/2}$, the unitary group $\mathrm{U}(n+m)$ is the submanifold of the general linear group $\mathrm{GL}(n+m,\mathbb{C})$   defined by the equation $P=I_{n+m}$.
	
	By Fact~6.6.3 in \cite{HKPV}, combining \eqref{U} with Fact~\ref{new new new measure decomposition for general linear group}, we obtain
	\begin{equation*}
		\begin{split}
			h(I_{n+m})\,d\mu_{\mathrm{U}(n+m)}(U)&\,\propto\,\frac{f(I_m)g(0_{m\times m},I_m)\prod_{k=1}^{n}(1-|z_k|^2)^{m-1}\prod_{1\leq i<j\leq n}|z_i-z_j|^2}{\det(\mathscr{A}_n^*\mathscr{A}_n+R_nR_n^*)^m\prod_{k=1}^{n-1}\det(\mathscr{A}_k^*\mathscr{A}_k+R_kR_k^*)}\\
			&\quad\cdot\bigwedge_{1\leq k\leq n}d\sigma(z_k)\bigwedge_{1\leq k\leq n}d\Theta(\omega_k)\bigwedge d\nu_n(V)\bigwedge d\mu_{\mathrm{U}(m)}(W),
		\end{split}
	\end{equation*}
	that is
	\begin{equation*}
		\begin{split}
			d\mu_{\mathrm{U}(n+m)}(U)&\,\propto\,\frac{\prod_{k=1}^{n}(1-|z_k|^2)^{m-1}\prod_{1\leq i<j\leq n}|z_i-z_j|^2}{\det(\mathscr{A}_n^*\mathscr{A}_n+R_nR_n^*)^m\prod_{k=1}^{n-1}\det(\mathscr{A}_k^*\mathscr{A}_k+R_kR_k^*)}\\
			&\quad\cdot d\mu_{\mathrm{U}(m)}(W)\,d\nu_n(V)\prod_{k=1}^{n}d\Theta(\omega_k)\prod_{k=1}^{n}d\sigma(z_k).
		\end{split}
	\end{equation*}
Here,   we use the fact $Q=I_{n+m}$,$G=0_{m\times m}$ and $\alpha_k=0$, $k=1,2,\cdots,n-1$, $D=I_m$ and $r_k^2=1-|z_k|^2$, $k=1,2,\cdots,n$, which are implied by $P=I_{n+m}$.

	For any $1\leq k\leq n$, since $Q=I_{n+m}$, it follows from \eqref{E} that
	\begin{equation*}
		\begin{split}
			\mathscr{A}_k^*\mathscr{A}_k+R_kR_k^*=I_k.
		\end{split}
	\end{equation*}
Therefore,
	\begin{equation*}
		\begin{split}
			\det\big(\mathscr{A}_k^*\mathscr{A}_k+R_kR_k^*\big)=1,
		\end{split}
	\end{equation*}
	which implies Fact~\ref{Haar measure decomposition}.

This completes the proof of Lemma~\ref{lemma about Haar measure}.
\end{proof}

We now turn to the Hua-Pickrell measure $\mu_{\mathrm{U}(n+m)}^{(\delta)}$ on the unitary group $\mathrm{U}(n+m)$, where $\mathrm{Re}\,\delta>-1/2$. Let us start with the following Lemma.

\begin{lemma}\label{lemma about unitary matrix}
	For $\begin{bmatrix}
		M_1&M_2\\
		M_3^*&M_4
	\end{bmatrix}\in\mathrm{U}(n+m)$,  suppose the $n\times n$ matrix  $M_1$ salsifies that $M_1$ and $I-M_1$ are invertible. Then the $m\times m$  matrix
	\begin{equation*}
		\begin{split}
			\big[I-M_3^*(I-M_1)^{-1}M_1^{*-1}M_3\big]\big[I+(M_3^*M_1^{-1})(M_3^*M_1^{-1})^*\big]^{-1/2}\in\mathrm{U}(m).
		\end{split}
	\end{equation*}
\end{lemma}

\begin{proof}[{\bf Proof of Lemma~\ref{lemma about unitary matrix}}]
	It is enough to prove that
	\begin{equation*}
		\begin{split}
			&\quad\,\,\,I+(M_3^*M_1^{-1})(M_3^*M_1^{-1})^*\\
			&=\big[I-M_3^*(I-M_1)^{-1}M_1^{*-1}M_3\big]^*\big[I-M_3^*(I-M_1)^{-1}M_1^{*-1}M_3\big].
		\end{split}
	\end{equation*}
	Since $\begin{bmatrix}
		M_1&M_2\\
		M_3^*&M_4
	\end{bmatrix}$ is unitary, we have
	\begin{equation*}
		\begin{split}
			M_1^*M_1+M_3M_3^*=I.
		\end{split}
	\end{equation*}
	Therefore,
	\begin{equation*}
		\begin{split}
			\mathrm{RHS}&=\big[I-M_3^*M_1^{-1}(I-M_1^*)^{-1}M_3\big]\big[I-M_3^*(I-M_1)^{-1}M_1^{*-1}M_3\big]\\
			&=I+M_3^*M_1^{-1}(I-M_1^*)^{-1}M_3M_3^*(I-M_1)^{-1}M_1^{*-1}M_3\\
			&\quad-M_3^*M_1^{-1}(I-M_1^*)^{-1}M_3-M_3^*(I-M_1)^{-1}M_1^{*-1}M_3\\
			&=I+M_3^*M_1^{-1}(I-M_1^*)^{-1}\big[M_3M_3^*-M_1^*(I-M_1)-(I-M_1^*)M_1\big](I-M_1)^{-1}M_1^{*-1}M_3\\
			&=I+M_3^*M_1^{-1}(I-M_1^*)^{-1}\big[M_3M_3^*+2M_1^*M_1-M_1^*-M_1\big](I-M_1)^{-1}M_1^{*-1}M_3\\
			&=I+M_3^*M_1^{-1}(I-M_1^*)^{-1}\big[I+M_1^*M_1-M_1^*-M_1\big](I-M_1)^{-1}M_1^{*-1}M_3\\
			&=I+M_3^*M_1^{-1}(I-M_1^*)^{-1}\big[(I-M_1^*)(I-M_1)\big](I-M_1)^{-1}M_1^{*-1}M_3\\
			&=I+M_3^*M_1^{-1}M_1^{*-1}M_3\\
			&=\mathrm{LHS}.
		\end{split}
	\end{equation*}
	This completes the proof of Lemma~\ref{lemma about unitary matrix}.
\end{proof}

With the help of Lemma~\ref{lemma about unitary matrix}, we have the following lemma about the Hua-Pickrell measure $\mu_{\mathrm{U}(n+m)}^{(\delta)}$ in the case of $\mathrm{Re}\,\delta>-1/2$.

\begin{lemma}\label{lemma about Hua-Pickrell measure}
	There exists a probability measure $\mathcal{T}_{(n,m)}^{(\delta)}$ on $\mathscr{T}(n,m)$ and a measurable transformation
	\begin{equation*}
		\begin{split}
			T^{(\delta)}:\big(\mathrm{U}(n+m),\mu_{\mathrm{U}(n+m)}^{(\delta)}\big)\to\big(\mathscr{T}(n,m),\mathcal{T}_{(n,m)}^{(\delta)}\big),
		\end{split}
	\end{equation*}
	which possess the following properties $:$
	
	$(\mathrm{i}).$ there is an $m\times m$ unitary matrix $H(z_1,\cdots,z_n,\omega_1,\cdots,\omega_n)\in\mathrm{U}(m)$ related to $z_1,\cdots,z_n$ and $\omega_1,\cdots,\omega_n$, such that
	\begin{equation*}
		\begin{split}
			&\quad\,\,d\mathcal{T}_{(n,m)}^{(\delta)}(W,V,\omega_1,\cdots,\omega_n,z_1,\cdots,z_n)\\
			&=C_{(n,m)}^{(\delta)}\big|\det(I-H(z_1,\cdots,z_n,\omega_1,\cdots,\omega_n)W)^{\delta}\big|^2\prod_{k=1}^{n}\big|(1-z_k)^{\delta}\big|^2\\
			&\quad\cdot\prod_{k=1}^{n}(1-|z_k|^2)^{m-1}\prod_{1\leq i<j\leq n}|z_i-z_j|^2\,d\mu_{\mathrm{U}(m)}(W)\,d\nu_n(V)\prod_{k=1}^{n}d\Theta(\omega_k)\prod_{k=1}^{n}d\sigma(z_k),
		\end{split}
	\end{equation*}
	where $C_{(n,m)}^{(\delta)}$ is the normalization constant $;$
	
	$(\mathrm{ii}).$ for $\mu_{\mathrm{U}(n+m)}^{(\delta)}$-a.e. $U\in\mathrm{U}(n+m)$,  the set $\{(\mathcal{P}T^{(\delta)}(U))_i:1\leqslant i\leqslant n\} $ is equal to the set $\lambda\big([U_{i,j}]_{1\leq i,j\leq n}\big)$ of  eigenvalues of the matrix $[U_{i,j}]_{1\leq i,j\leq n};$
	
	$(\mathrm{iii}).$ for any integrable function $\varphi\in L^1\big(\mathscr{T}(n,m),\mathcal{T}_{(n,m)}^{(\delta)}\big)$,
	\begin{equation*}
		\begin{split}
			\int_{\mathrm{U}(n+m)}\varphi\circ T^{(\delta)}\,d\mu_{\mathrm{U}(n+m)}^{(\delta)}=\int_{\mathscr{T}(n,m)}\varphi\,d\mathcal{T}_{(n,m)}^{(\delta)}.
		\end{split}
	\end{equation*}
\end{lemma}

\begin{proof}[{\bf Proof of Lemma~\ref{lemma about Hua-Pickrell measure}}]
We will divide the proof of Lemma~\ref{lemma about Hua-Pickrell measure} into two steps.

\textbf{Step I.}
   We have the following measure decomposition for the Hua-Pickrell measure $\mu_{\mathrm{U}(n+m)}^{(\delta)}$.
	
\begin{fact}\label{Hua-Pickrell measure decomposition}
	In the notations of the proof of Lemma~\ref{lemma about Haar measure}, the Hua-Pickrell measure $\mu_{\mathrm{U}(n+m)}^{(\delta)}$ on $\mathrm{U}(n+m)$ has the following decomposition:
	\begin{equation*}
		\begin{split}
			d\mu_{\mathrm{U}(n+m)}^{(\delta)}(U)&\,\propto\,\left|\det\Big(I-\big[I-R_n^*(I-\mathscr{A}_n)^{-1}\mathscr{A}_n^{*-1}R_n\big]\big[I+(R_n^*\mathscr{A}_n^{-1})(R_n^*\mathscr{A}_n^{-1})^*\big]^{-1/2}W\Big)^{\delta}\right|^2\\
			&\quad\cdot\prod_{k=1}^{n}\big|(1-z_k)^{\delta}\big|^2\prod_{k=1}^{n}(1-|z_k|^2)^{m-1}\prod_{1\leq i<j\leq n}|z_i-z_j|^2\\
			&\quad\cdot d\mu_{\mathrm{U}(m)}(W)\,d\nu_n(V)\prod_{k=1}^{n}d\Theta(\omega_k)\prod_{k=1}^{n}d\sigma(z_k).
		\end{split}
	\end{equation*}
\end{fact}

Proof of Fact~\ref{Hua-Pickrell measure decomposition}:
	Recall that
	\begin{equation}\label{V}
		\begin{split}
			d\mu_{\mathrm{U}(N)}^{(\delta)}(U)\,\propto\,\big|\det(I-U)^{\delta}\big|^2d\mu_{\mathrm{U}(N)}(U).
		\end{split}
	\end{equation}
 Using  \eqref{modified partitioned matrix}, we have
	\begin{equation*}
		\begin{split}
			\det(I-U)&=\det\left(I-
			\begin{bmatrix}
				V&0\\
				0&I
			\end{bmatrix}
			\begin{bmatrix}
				Z+T&S\\
				R^*&A
			\end{bmatrix}
			\begin{bmatrix}
				V^*&0\\
				0&I
			\end{bmatrix}\right)\\
			&=\det
			\begin{bmatrix}
				I-(Z+T)&-S\\
				-R^*&I-A
			\end{bmatrix}\\
			&=\det\big[I-(Z+T)\big]\det\Big(I-A-R^*\big[I-(Z+T)\big]^{-1}S\Big)\\
			&=\det\big[I-A-R_n^*(I-\mathscr{A}_n)^{-1}S\big]\prod_{k=1}^{n}(1-z_k).
		\end{split}
	\end{equation*}
	Noting that $Q^{(2)}=0_{n\times m}$ and $D=I_m$, and by \eqref{H} and \eqref{I}, we get
	\begin{equation*}
		\begin{split}
			&\quad\,\det\big[I-A-R_n^*(I-\mathscr{A}_n)^{-1}S\big]\\
			&=\det\Big(I-\big[I-R_n^*(I-\mathscr{A}_n)^{-1}\mathscr{A}_n^{*-1}R_n\big]\big[I+(R_n^*\mathscr{A}_n^{-1})(R_n^*\mathscr{A}_n^{-1})^*\big]^{-1/2}W\Big).
		\end{split}
	\end{equation*}
	It follows that
	\begin{equation}\label{W}
		\begin{split}
			&\quad\,\,\big|\det(I-U)^{\delta}\big|^2\\
			&=\left|\det\Big(I-\big[I-R_n^*(I-\mathscr{A}_n)^{-1}\mathscr{A}_n^{*-1}R_n\big]\big[I+(R_n^*\mathscr{A}_n^{-1})(R_n^*\mathscr{A}_n^{-1})^*\big]^{-1/2}W\Big)^{\delta}\right|^2\\
			&\quad\cdot\prod_{k=1}^{n}\big|(1-z_k)^{\delta}\big|^2.
		\end{split}
	\end{equation}
	Combining \eqref{V}, \eqref{W} with Fact~\ref{Haar measure decomposition}, we obtain Fact~\ref{Hua-Pickrell measure decomposition}.

\textbf{Step II.}
Denote the $m\times m$ matrix
\begin{equation*}
	\begin{split}
		&\quad\,\,H(z_1,z_2,\cdots,z_n,\omega_1,\omega_2,\cdots,\omega_n)\\
		&=\big[I-R_n^*(I-\mathscr{A}_n)^{-1}\mathscr{A}_n^{*-1}R_n\big]\big[I+(R_n^*\mathscr{A}_n^{-1})(R_n^*\mathscr{A}_n^{-1})^*\big]^{-1/2}.
	\end{split}
\end{equation*}

Since that $\begin{bmatrix}
	Z+T&S\\
	R^*&A
\end{bmatrix}=\begin{bmatrix}
			V^*&0\\
			0&I
		\end{bmatrix} U \begin{bmatrix}
			V&0\\
			0&I
		\end{bmatrix}$ is a unitary matrix,  Lemma~\ref{lemma about unitary matrix}  implies that $H(z_1,z_2,\cdots,z_n,\omega_1,\omega_2,\cdots,\omega_n)\in\mathrm{U}(m)$ immediately.

This completes the proof of Lemma~\ref{lemma about Hua-Pickrell measure}.
\end{proof}

We now turn to prove Theorem~\ref{thm-Eigenvalues of truncations}.

\begin{proof}[{\bf Proof of Theorem~\ref{thm-Eigenvalues of truncations}}]
Let $U$ be a random unitary matrix drawn from the Hua-Pickrell distribution $\mu_{\mathrm{U}(n+m)}^{(\delta)}$ on $\mathrm{U}(n+m)$, where $\mathrm{Re}\,\delta>-1/2$. We are going to derive the density of eigenvalues of the truncated unitary matrix $[U_{i,j}]_{1\leq i,j\leq n }$, i.e. the joint density of vector $(z_1,z_2,\cdots,z_n)$ in uniform random order.

By Fact~\ref{Hua-Pickrell measure decomposition}, all we need is to calculate the integral
\begin{equation}\label{X}
	\begin{split}
		\int_{\mathrm{U}(m)\times(S^{2m-1})^n} \Big|\det\big[I-H(z_1,\cdots,z_n,\omega_1,\cdots,\omega_n)W\big]^{\delta}\Big|^2d\mu_{\mathrm{U}(m)}(W)\prod_{k=1}^{n}d\Theta(\omega_k),
	\end{split}
\end{equation}
where the $m\times m$ unitary matrix $H(z_1,z_2,\cdots,z_n,\omega_1,\omega_2,\cdots,\omega_n)\in\mathrm{U}(m)$ related to $z_1,z_2,\cdots,z_n$ and $\omega_1,\omega_2,\cdots,\omega_n$.

Because the Haar measure on $\mathrm{U}(m)$ is invariant under left multiplication by unitary matrices, the integral
\begin{equation*}
	\begin{split}
		\int_{\mathrm{U}(m)} \Big|\det\big[I-H(z_1,\cdots,z_n,\omega_1,\cdots,\omega_n)W\big]^{\delta}\Big|^2\,d\mu_{\mathrm{U}(m)}(W)
	\end{split}
\end{equation*}
is a constant, independent of $z_1,z_2,\cdots,z_n,\omega_1,\omega_2,\cdots,\omega_n$. And then the integral~\eqref{X} is independent of $z_1,z_2,\cdots,z_n$.

Therefore, we conclude from Fact~\ref{Hua-Pickrell measure decomposition} that the joint density of vector $(z_1,z_2,\cdots,z_n)$ is proportional to
\begin{equation*}
	\begin{split}
		\prod_{k=1}^{n}\big|(1-z_k)^{\delta}\big|^2\prod_{k=1}^{n}(1-|z_k|^2)^{m-1}\prod_{1\leq i<j\leq n}|z_i-z_j|^2\prod_{k=1}^{n}d\sigma(z_k).
	\end{split}
\end{equation*}
Writing that
\begin{equation*}
	\begin{split}
		d\mu^{(m,\delta)}(z)=|(1-z)^{\delta}|^2(1-|z|^2)^{m-1}d\sigma(z),
	\end{split}
\end{equation*}
we can restate that the joint density of the vector $(z_1,z_2,\cdots,z_n)$ is proportional to
\begin{equation}\label{Y}
	\begin{split}
		\prod_{1\leq i<j\leq n}|z_i-z_j|^2\prod_{k=1}^{n}d\mu^{(m,\delta)}(z_k).
	\end{split}
\end{equation}
This is an orthogonal polynomial ensemble by \eqref{OPE}.

It follows from \eqref{Y} that the eigenvalues of the truncated unitary matrix $[U_{i,j}]_{1\leq i,j\leq n }$ form  a determinantal point process $\mathscr{X}_n^{(m,\delta)}$ on the unit disc $\mathbb{D}$ with the kernel
\begin{equation*}
	\begin{split}
		K_n^{(m,\delta)}(z,w)=\sum_{k=0}^{n-1}P_k^{(m,\delta)}(z)\overline{P_k^{(m,\delta)}(w)}
	\end{split}
\end{equation*}
with respect to the reference measure $\mu^{(m,\delta)}$. Here $\{P_k^{(m,\delta)}\}_{k=0}^{\infty}$ is the family of orthonormal polynomials obtained by applying Gram-Schmidt orthogonalization procedure to $\{z^k\}_{k=0}^{\infty}$ in $L^2(\mathbb{D},\mu^{(m,\delta)})$.

This completes the proof of Theorem~\ref{thm-Eigenvalues of truncations}.
\end{proof}

\section{Proof of Theorem~\ref{thm-Limiting point process}}
This section is devoted to the proof of Theorem~\ref{thm-Limiting point process}. Let us denote the reproducing kernel Hilbert space
\begin{equation*}
	\begin{split}
		L_a^2(\mathbb{D},\mu^{(m,\delta)})=\Big\{f:\mathbb{D}\to\mathbb{C}\,\Big|\,f\,\text{is holomorphic and}\,\int_{\mathbb{D}}|f(z)|^2d\mu^{(m,\delta)}(z)<\infty\Big\}.
	\end{split}
\end{equation*}
\begin{lemma}\label{lemma-4}
	The linear subspace  $\mathrm{span}\{z^k:k=0,1,2,\cdots\} $ is dense in $L_a^2(\mathbb{D},\mu^{(m,\delta)})$.
\end{lemma}

\begin{proof}
Let  $L_a^2(\mathbb{D},\mu^{(m)})$ be
 the standard weighted Bergman space
\begin{equation*}
	\begin{split}		L_a^2(\mathbb{D},\mu^{(m)})=\Big\{f:\mathbb{D}\to\mathbb{C}\,\Big|\,f\,\text{is holomorphic and}\,\int_{\mathbb{D}}|f(z)|^2d\mu^{(m)}(z)<\infty\Big\},
	\end{split}
\end{equation*}
where the measure  $d\mu^{(m)}(z)=(1-|z|^2)^{m-1}d\sigma(z).$
It suffices to show that the linear subspace $\mathrm{span}\{(1-z)^{\delta}z^k:k=0,1,2,\cdots\} $ is dense in $L_a^2(\mathbb{D},\mu^{(m)})$, or
\begin{equation}\label{1 in span}
	\begin{split}
		1\in\overline{\mathrm{span}}^{L_a^2(\mathbb{D},\mu^{(m)})}\{(1-z)^{\delta}z^k:k =0,1,2,\cdots\}.
	\end{split}
\end{equation}

In the case $\mathrm{Re}\,\delta\geq0$, one sees that $\sup_{z\in\mathbb{D}}|(1-z)^{\delta}|<\infty$. For any $0<r<1$, since 
there exists a sequence of polynomials $\{p_k^{(r)}\}$  which    converges to $\frac{1}{(1-rz)^{\delta}}$ uniformly on $\overline{\mathbb{D}}$, we have
\begin{equation*}
	\begin{split}
		\Big(\frac{1-z}{1-rz}\Big)^{\delta}\in\overline{\mathrm{span}}^{L_a^2(\mathbb{D},\mu^{(m)})}\{(1-z)^{\delta}z^k:k=0,1,2,\cdots\}.
	\end{split}
\end{equation*}
Note $\big|\frac{1-z}{1-rz}\big|\leq\big|\frac{1-r}{1-rz}\big|+\big|\frac{r-z}{1-rz}\big|\leq 2$. By Lebesgue's dominated convergence theorem, we have
\begin{equation*}
	\begin{split}
		\lim_{r\to1^-}\int_{\mathbb{D}}\Big|\Big(\frac{1-z}{1-rz}\Big)^{\delta}-1\Big|^2d\mu^{(m)}(z)=0.
	\end{split}
\end{equation*}
This implies that \eqref{1 in span}.

In the case $-1/2<\mathrm{Re}\,\delta<0$, $(1-z)^{-\delta}$ is a holomorphic function in the unit disc $\mathbb{D}$ and continuous to the unit circle $\mathbb{T}$. Hence by Mergelyan's theorem, there is a sequence of polynomials $\{q_k\}$ such that $q_k$ converges to $(1-z)^{-\delta}$ uniformly on $\overline{\mathbb{D}}$. Because of
\begin{equation*}
	\begin{split}
		\int_{\mathbb{D}}\big|(1-z)^{\delta}q_k(z)-1\big|^2d\mu^{(m)}(z)&=\int_{\mathbb{D}}\big|q_k(z)-(1-z)^{-\delta}\big|^2\big|(1-z)^{\delta}\big|^2d\mu^{(m)}(z)\\
		&\leq\left\| q_k(z)-(1-z)^{-\delta}\right\|_\infty^2\mu^{(m,\delta)}(\mathbb{D}),
	\end{split}
\end{equation*}
we also obtain \eqref{1 in span}.

This completes the proof of Lemma~\ref{lemma-4}.
\end{proof}

Recall that the correlation kernel of the determinantal point process $\mathscr{X}_n^{(m,\delta)}$ established in Theorem~\ref{thm-Eigenvalues of truncations} is
\begin{equation*}
	\begin{split}
		K_n^{(m,\delta)}(z,w)=\sum_{k=0}^{n-1}P_k^{(m,\delta)}(z)\overline{P_k^{(m,\delta)}(w)},
	\end{split}
\end{equation*}
where $\{P_k^{(m,\delta)}\}_{k=0}^{\infty}$ is the family of orthonormal polynomials obtained by applying Gram-Schmidt orthogonalization procedure to $\{z^k\}_{k=0}^{\infty}$ in $L^2(\mathbb{D},\mu^{(m,\delta)})$. By Lemma~\ref{lemma-4}, the limiting kernel taken by $n\to\infty$ of $K_n^{(m,\delta)}$ is the reproducing kernel of $L_a^2(\mathbb{D},\mu^{(m,\delta)})$.

It is easy to verify $L_a^2(\mathbb{D},\mu^{(m,\delta)})=(1-z)^{-\delta}L_a^2(\mathbb{D},\mu^{(m)})$. Thus by \cite[Propositon~5.20]{PR}, we obtain the following lemma:
\begin{lemma}\label{lemma-5}
	The reproducing kernel of $L_a^2(\mathbb{D},\mu^{(m,\delta)})$ is
	\begin{equation*}
		\begin{split}
			K^{(m,\delta)}(z,w)=\frac{m}{\pi(1-z)^{\delta}(1-z\overline w)^{m+1}(1-\overline w)^{\bar\delta}}.
		\end{split}
	\end{equation*}
\end{lemma}

It follows from Lemma~\ref{lemma-5} that the limiting point process $\mathscr{X}^{(m,\delta)}$, taken by $n\to\infty$ of $\mathscr{X}_n^{(m,\delta)}$, is the determinantal point process on $\mathbb{D}$ with kernel $K^{(m,\delta)}$ with respect to the reference measure $\mu^{(m,\delta)}$.

Recall that the determinantal point process $\mathscr{X}^{[m]}$ introduced in subsection~\ref{Truncations of Haar unitary matrices}, which is a determinantal point process on the unit disc $\mathbb{D}$ with weighted Bergman kernel
\begin{equation*}
	\begin{split}
		K^{[m]}(z,w)=\frac{1}{(1-z\overline w)^{m+1}}
	\end{split}
\end{equation*}
with respect to the reference measure $d\mu^{[m]}(z)=\frac{m}{\pi}(1-|z|)^{m-1}d\sigma(z)$. It is worth noticing that for any $n\in\mathbb{N_+}$,
\begin{equation*}
	\begin{split}
		\det\big[K^{(m,\delta)}(z_i,z_j)\big]_{1\leq i,j\leq n}\prod_{k=1}^{n}d\mu^{(m,\delta)}(z_k)=\det\big[K^{[m]}(z_i,z_j)\big]_{1\leq i,j\leq n}\prod_{k=1}^{n}d\mu^{[m]}(z_k).
	\end{split}
\end{equation*}
Hence by definition~\eqref{DPP-definition}, for any $\delta\in\mathbb{C}$ satisfying $\mathrm{Re}\,\delta>-1/2$, the determinantal point process $\mathscr{X}^{(m,\delta)}$ is always equal to $\mathscr{X}^{[m]}$, independent of $\delta$.

This completes the proof of Theorem~\ref{thm-Limiting point process}.

\vskip3mm
  {\bf Conflict of Interest Statement}. The authors state that there is no conflict of interest.

\vskip3mm
   {\bf Data Availability Statement}. Data sharing not applicable to this article as no datasets were generated or analyzed during the current study.

\end{document}